\documentclass{amsart}
\usepackage{amsmath,amsthm}
\usepackage{amsfonts,amssymb}
\usepackage{accents}
\usepackage{enumerate}
\usepackage{accents,color}
\usepackage{graphicx}

\newcommand{\lvt}{\left|\kern-1.35pt\left|\kern-1.3pt\left|}
\newcommand{\rvt}{\right|\kern-1.3pt\right|\kern-1.35pt\right|}

\newtheorem{thm}{Theorem}[section]
\newtheorem{cor}[thm]{Corollary}
\newtheorem{lem}[thm]{Lemma}
\newtheorem{prop}[thm]{Proposition}

\theoremstyle{remark}

 \def\la{{\langle}}
 \def\ra{{\rangle}}

 \def\i{{\mathrm{i}}} 
 \def\dd{{\mathrm{d} }} 
  
 \def\a{{\alpha}}
 \def\b{{\beta}}
 
 \def\k{{\kappa}}
 \def\t{{\theta}}
 \def\l{{\lambda}}
 \def\d{{\delta}}
 
 \def\s{\sigma}
 \def\la{{\langle}}
 \def\ra{{\rangle}}

 \def\CH{{\mathcal H}}

 \def\CP{{\mathcal P}}

 \def\NN{{\mathbb N}}

 \def\RR{{\mathbb R}}
  \def\SS{{\mathbb S}}
 \def\ZZ{{\mathbb Z}}

\def\lla{\langle{\kern-2.5pt}\langle}      
\def\rra{\rangle{\kern-2.5pt}\rangle}      

\newcommand{\wt}{\widetilde}
\newcommand{\wh}{\widehat}

\def\f{\frac}

\graphicspath{{./}}
\begin{document}
 
\title{Intertwining operators associated to dihedral groups}

\author{Yuan Xu}
\address{Department of Mathematics\\ University of Oregon\\ 
Eugene, Oregon 97403-1222.}\email{yuan@uoregon.edu}

\thanks{The author was supported in part by NSF Grant DMS-1510296.}

\date{\today}
\keywords{Intertwining operator, Dunkl operators, Dihedral group, orthogonal polynomials, generating function}
\subjclass[2010]{33C45, 44A20; Secondary 33C50, 33C80}

\begin{abstract} 
The Dunkl operators associated to a dihedral group are a pair of differential-difference operators that generate a 
commutative algebra acting on differentiable functions in $\RR^2$. The intertwining operator intertwines 
between this algebra and the algebra of differential operators. The main result of this paper is an integral 
representation of the intertwining operator on a class of functions. As an application, closed formulas for 
the Poisson kernels of $h$-harmonics and sieved Gegenbauer polynomials are deduced when one of the 
variables is at vertices of a regular polygon, and similar formulas are also derived for several other related
families of orthogonal polynomials. 
\end{abstract}

\maketitle

\section{Introduction}
\setcounter{equation}{0}

Let $G$ be a reflection group with a fixed positive root system $R_+$. Let $v \mapsto \k_v$ be a 
nonnegative multiplicity function defined on $R_+$ with the property that it is a constant on each conjugate 
class of $G$. Then the Dunkl operators \cite{D89} are defined by 
\begin{equation} \label{eq:Dunkl-Op}
    D_i f(x) = \partial_i  f(x) + \sum_{v \in R_+} \k_v \frac{f(x) - f (x \s_v)}{\la x,\s_v \ra} v_i,  \qquad i =1,2 , \ldots,d,
\end{equation}
where $x \s_v := x - 2\la x, v \ra v /\|v\|^2$ and $\la x,\s_v\ra$ is the dot product in $\RR^d$. These first order 
differential-difference operators commute in the 
sense that $D_i D_j = D_j D_i$ for $1 \le i, j \le d$. A linear operator, denoted by $V_\kappa$, that satisfies
the relations 
\begin{equation} \label{eq:Vk-defn}
    D_i V_\k  = V_\k \partial_i, \qquad 1 \le i \le d,
\end{equation}
is called an intertwining operator, which is uniquely determined if it also satisfies $V_\k 1 = 1$ and 
$V_\k \CP_n^d \subset \CP_n^d$, where $\CP_n^d$ denotes the space of homogeneous polynomials 
of degree $n$ in $d$ variables. 

The remarkable commuting property of the Dunkl operators allows a far reaching generalization of classical
analysis from $L^2$ with respect to the Lebesgue measure to weighted $L^2$ space on either $\RR^d$ or the 
unit sphere $\SS^{d-1}$, where the weight function, multiplied by $e^{-\|x\|^2}$ on $\RR^d$, is defined by
\begin{equation} \label{eq:h_k-general}
       h_\k(x) = \prod_{v\in R_+} |\la x , v\ra|^{\k_v}.
\end{equation}
The intertwining operator plays an essential role in the generalization. For example, the weighted Fourier 
transform has $V_\k [e^{i \la \cdot, x\ra}](y)$ in places of $e^{i \la x,y\ra}$ and the zonal harmonics in the weighted
setting is given by $V_\k [ C_n^{\l_\k} (\la x,\cdot\ra)](y)$, where $C_n^\l$ denotes the Gegenbauer
polynomial and $\l_\k = \sum_v \k_v + \frac{d-2}{2}$. 

In recent years, the Fourier analysis associated with reflection groups has attracted considerable attention; see,
for example, \cite{DaX, DX} and references therein. A large portion of the classical Fourier analysis has been 
extended to the weighted setting. However, much of finer analysis that relies essentially on reflection symmetry 
requires detail knowledge of the intertwining operator. This is the reason that finer analysis has been carried out 
up to now only in the case of $G= \ZZ_2^d$, for which the intertwining operator is given explicitly as an integral 
operator.

There have been several attempts for finding an integral representation of $V_\k f$ for other reflection groups. 
Partial results have been obtained for the symmetric group $S^3$ \cite{D95} and the dihedral group $I_4$ 
\cite{D-B2, X00}, but the results in these works are not as satisfactory since the weight function in the integral 
operator may not be nonnegative and, as a consequence, the positivity of the integral is not evident. In the $I_4$ 
case, the results are established for polynomials, not verified directly via \eqref{eq:Vk-defn}. 

In the present paper we consider the intertwining operator associated to dihedral groups. Let $I_k$ denote the 
dihedral group defined as the symmetric group of the $k$-th regular polygon. We choose the positive root system
$R_+ = \{v_j: 0 \le j \le k-1\}$ by 
\begin{equation} \label{eq:vj}
v_j = \left (\sin \left(\tfrac{j \pi}{k}\right),  - \cos \left(\tfrac{j \pi}{k}\right) \right).
\end{equation}
The reflection $x\s_j$ of $x = (x_1,x_2)$ in $v_j$ is given by 
\begin{equation} \label{eq:sj}
x \s_j = \left(\cos \left(\tfrac{2 j\pi}{k}\right) x_1 + \sin \left(\tfrac{2 j\pi}{k}\right) x_2, \, 
    \sin \left(\tfrac{2 j\pi}{k}\right) x_1-\cos \left(\tfrac{2 j\pi}{k}\right) x_2 \right).
\end{equation}
We consider the case that the multiplicity function $\k$ is a constant, which we denote by $\l$. In this setting, 
the Dunkl operators are given by
\begin{align}\label{eq:Dj}
\begin{split}
 D_1 f(x) & = \frac{\partial f}{\partial x_1} +  \l \sum_{j=0}^{k-1}  \frac{f(x) - f(x \s_j)}{\la x, v_j\ra} \sin  \left(\frac{j \pi}{k}\right),  \\
 D_2 f(x) & = \frac{\partial f}{\partial x_2} -  \l \sum_{j=0}^{k-1}  \frac{f(x) - f(x \s_j)}{\la x, v_j\ra}  \cos \left(\frac{j \pi}{k}\right).
\end{split}
\end{align}

Throughout this paper, we let $T^{k-1}$ denote the simplex defined by 
$$
  T^{k-1} := \{ u \in \RR^{k-1}:   u_1 \ge 0, \ldots, u_{k-1} \ge 0, \, \,  u_1+ \cdots + u_{k-1} \le 1\}.
$$
Our main result gives an integral representation for the intertwining operator on a class of functions. 

\begin{thm} \label{thm:main}
Let $f$ be a differentiable function on $\RR$ so that the integral in \eqref{eq:main} is finite. For $0 \le p \le 2k-1$, 
define 
$$
   F_p (x_1,x_2 ):= f \left(\cos  \left(\frac{ p \pi}{k}\right)x_1 + \sin \left(\frac{ p \pi}{k}\right) x_2 \right). 
$$
Then, for $k =2, 3, 4, \ldots$, the intertwining operator $V_\l$ for the dihedral group $I_k$ with one parameter $\l$ satisfies 
\begin{align} \label{eq:main}
 V_\l F_p(x_1,x_2) =&\,a_\l^{(k)} \int_{T^{k-1}} f \left ( \cos  \left(\frac{ p \pi}{k}\right) \big(c(u) x_1 + s(u) x_2\big) 
                               + \sin  \left(\frac{ p \pi}{k}\right) \big (c(u) x_2 - s(u) x_1\big) \right) \notag \\  
                               &    \qquad \qquad \qquad \qquad \qquad   \times u_0^\l \prod_{i=1}^{k-1} u_i^{\l-1}\dd u, 
\end{align}
where $u_0 := 1-u_1-\cdots - u_{k-1}$, 
$$
  c(u):= \sum_{j=0}^{k-1} \cos \left(\frac{2 j\pi}{k}\right) u_j \quad \hbox{and} \quad  
    s(u):= \sum_{j=0}^{k-1} \sin\left(\frac{2 j\pi}{k}\right) u_j
$$
and $a_\l^{(k)}$ is chosen so that $V_\l 1 =1$ or, given explicitly,  
$$
   a_\l^{(k)} =  \frac{\l \Gamma(\l)^k}{\Gamma( k\l+1)}.
$$
\end{thm}

Although the result does not give a full integral representation of the intertwining operator for the dihedral group, 
the explicit formula of the integral in \eqref{eq:main} is notable and suggestive for a possible final form. There 
is little methodology for identifying an integral transform that satisfies \eqref{eq:Vk-defn}, the discovery of our 
\eqref{eq:main} is motivated by an integral formula in \cite{X15}, see Lemma \ref{lem:integral} below, and is 
the result of trial and error, starting from $I_4$. Once the formula is identified, a proof can be given by a direct 
verification of \eqref{eq:Vk-defn}. 

As an application, we obtain a closed form formula for the Poisson kernel of the $h$-harmonics associated to 
the dihedral group $I_k$ when one of the variable is at vertices of a regular $k$-gon. With one parameter,
the $h$-harmonics associated with $I_k$ can be written in terms of the sieved Gegenbauer polynomials 
studied in \cite{AAA}, which are orthogonal polynomials on $[-1,1]$ with respect to the weight funciton
\begin{equation} \label{eq:weightSG}
  w_{\l \pm \f12}^{(k)}(t) =  |U_{k-1}(t)|^{2 \l} (1-t^2)^{\l \pm \f12}, \qquad k =1,2, 3, \ldots. 
\end{equation}
Our result leads to a closed form formula for the Poisson kernels of the sieved Gegenbauer polynomials, which
can also be used to derive properties of these polynomials. Furthermore, $h$-harmonics can also be related,
when $k$ is even, to another family of polynomials, which leads us to study orthogonal polynomials with 
respect to 
\begin{equation} \label{eq:weightSG2}
 w_{\l \pm \f12}^{(k), \pm 1}(t) = (1\pm t) |U_{k-1}(t)|^{2 \l} (1-t^2)^{\l-\f12}, \qquad k =1,2, 3, \ldots, 
\end{equation} 
for which we can also derive the closed formulas for their Poisson kernels.

The paper is organized as follows. In the next section, we recall essential results on the $h$-harmonics associated
to the dihedral group. The main result, Theorem \ref{thm:main}, is proved in Section 3. As applications of the main 
result, we derive closed form formulas for $h$-harmonics and sieved Gegenbauer polynomials in Section 4,
explicit bases of orthogonal polynomials in Section 5, and study orthogonal polynomials for the weight functions
\eqref{eq:weightSG2} in Section 6. Finally, in Section 7, we discuss product formulas of orthogonal polynomials
and its relation with the intertwining operator. 

\section{Dihedral symmetry}
\setcounter{equation}{0}

The dihedral group $I_k$ is the symmetric group of regular polygon of $k$ sides. We consider the case
of even $k$ and odd $k$ separately. 

\subsection{Dihedral group $I_{2k}$}
For the dihedral group $I_{2k}$, we choose the positive root system as 
$$
   v_j = \big(\sin (\tfrac{j \pi}{2k}), - \cos (\tfrac{j \pi}{2k})\big), \qquad j=0,1,\ldots, 2 k -1. 
$$
For $x = r (\cos \t, \sin \t)$, we have $\la x, v_j \ra = r\sin (\tfrac{j \pi}{2k} -\t)$. The reflection $\s_j$ is given by
$$ 
  x \s_j =  x - 2 \la x, v_j \ra v_j = (\cos(\tfrac{j \pi}{k} - \t), \sin (\tfrac{j \pi}{k} -\t)). 
$$
The group has two conjugacy classes, described by $v_{2j}$ and $v_{2j+1}$, respectively.  The Dunkl operator 
associated with the dihedral group with parameters $\l \ge 0$ and $\mu \ge 0$ are given by
\begin{align*}
 D_1f(x) &= \frac{\partial f}{\partial x_1} + \l \sum_{j=0}^{k-1} \frac{f(x) - f(x \s_{2j})}{\la x, v_{2j}\ra} \sin \left(\tfrac{j \pi}{k}\right) 
   + \mu \sum_{j=0}^{k-1} \frac{f(x) - f(x \s_{2j+1})}{\la x, v_{2j+1}\ra} \sin\left(\tfrac{(2j+1)\pi}{2k}\right), \\
 D_2f(x) & =\frac{\partial f}{\partial x_2} - \l \sum_{j=0}^{k-1} \frac{f(x) - f(x \s_{2j})}{\la x, v_{2j}\ra} \cos \left(\tfrac{j \pi}{k}\right)  
    - \mu \sum_{j=0}^{k-1} \frac{f(x) - f(x \s_{2j+1})}{\la x, v_{2j+1}\ra} \cos\left(\tfrac{(2j+1)\pi}{2k}\right).
\end{align*}
In this case, the weight function $h_\k$ in \eqref{eq:h_k-general} is given by 
$$
     h_{\l,\mu}^{(2 k)} (x_1,x_2) = r^{\l+\mu} |\sin (k \t)|^{\l} |\cos (k \t)|^{\mu}
$$
in polar coordinates $(x_1,x_2) = (r\cos \t, r \sin\t)$. In particular, for $I_2 = \ZZ_2^2$, the weight function is 
$h_{\l,\mu}^{(2)} (x_1,x_2) = |x_1|^\mu |x_2|^\l$ and, with $k=1$, 
\begin{align*}
 D_1f(x) &= \frac{\partial f}{\partial x_1} + \mu \frac{f(x) - f(-x_1,x_2)}{x_1}, \\
 D_2f(x) & =\frac{\partial f}{\partial x_2} + \l \frac{f(x) - f(x_1, - x_2)}{x_2}.
\end{align*}
The intertwining operator $V_{\l,\mu}$ for $I_2$ is given by an integral transform \cite[p. 232]{DX}
\begin{equation}\label{eq:intertwZ2}
 V_{\l,\mu}f(x_1,x_2) = c_\l c_\mu \int_{-1}^1 \int_{-1}^1 f(s x_1,t x_2) (1+s)(1-s^2)^{\mu-1}  (1+t)(1-t^2)^{\l-1} \dd s \dd t.
\end{equation}
No other satisfactory integral representation of the intertwining operator is known for other dihedral groups. 

Let $\CP_n^2$ be the space of homogeneous polynomials of two variables of degree $n$. A polynomial
$Y \in \CP_n^2$ is called an $h$-harmonics if $\Delta_h Y =0$ where $\Delta_h = D_1^2 + D_2^2$. 
Let $\CH_n(h_{\l,\mu}^{(2k)})$ be the space of $h$-harmonics associated to the dihedral group $I_k$. It is 
known that $h$-harmonics of different degrees are orthogonal; that is, 
\begin{equation} \label{eq:h-ortho}
  \int_{\SS^1} Y_n (\xi) Y_m(\xi) \left[ h_{\l,\mu}^{(2 k)}(\xi)\right]^2 \dd \s(\xi) =0, \qquad  \quad Y_n \in \CH_n(h_{\l,\mu}^{(2k)}). 
\end{equation}
As in the case of ordinary spherical harmonics, we know that $\dim \CH_0(h_{\l,\mu}^{(2k)}) = 1$ and 
$\dim \CH_n(h_{\l,\mu}^{(2k)}) = 2$ for $n \ge 1$. An orthogonal basis of $\CH_n(h_{\l,\mu}^{(2k)})$ can be given 
explicitly in terms of the Jacobi polynomials $P_n^{(\a,\b)}$. Let us define first the generalized Gegenbauer 
polynomials $C_n^{(\l,\mu)}$ by 
\begin{align}\label{eq:GGpoly}
\begin{split}
 C_{2n}^{(\l,\mu)} (t) & = \frac{(\l+\mu)_n}{(\mu+\f12)_n} P_n^{(\l-\f12,\mu-\f12)}(2 t^2-1), \\
  C_{2n+1}^{(\l,\mu)} (t) & = \frac{(\l+\mu)_{n+1}}{(\mu+\f12)_{n+1}}\, t  P_n^{(\l-\f12,\mu+\f12)}(2 t^2-1), 
\end{split}
\end{align}
which are orthogonal with respect to the weigh function 
$$
 w(t) = |t|^{2\mu}(1-t^2)^{\l-\f12}, \qquad t \in [ -1,1]. 
$$
 
In the polar coordinates $x_1 = r \cos \t$ and $x_2 = r \sin \t$, an $h$-harmonic $Y$ can be written as 
$Y(x_1,x_2) = r^n \wt Y(\t)$, where $\wt Y(\t) = Y(\cos \t, \sin \t)$. We state an explicit orthogonal basis 
for $\CH_n(h_{\l,\mu}^{(2k)})$ in $\wt Y_{n,1}$ and $\wt Y_{n,2}$ \cite{D92}. 

\begin{prop} \label{prop:basis}
For $n = m k + j$ with $0 \le j \le k-1$,  define
\begin{align} \label{eq:basis}
\begin{split}
\wt Y_{mk + j, 1}(\t) & = \frac{n+2\l+\delta_n}{2\l+2\mu} \cos j \t \, C_m^{(\l,\mu)}(\cos k \t) 
    - \sin j \t \sin k \t  \, C_{m-1}^{(\l+1,\mu)}(\cos k\t),  \\
\wt Y_{m k+ j, 2}(\t) & = \frac{n+2\l+\delta_n}{2\a+2\mu} \sin j \t \, C_m^{(\l,\mu)}(\cos k \t)
     - \cos j \t \sin k \t \, C_{m-1}^{(\l+1,\mu)}(\cos k \t),
\end{split}
\end{align}
where $\d_n = 2\b$ if $k$ is even and $\d_n = 0$ if $k$ is odd. 
Then $\{Y_{n,1}, Y_{n,2}\}$ is an orthogonal basis of $\CH_n(h_{\l,\mu}^{(2k)})$. 
\end{prop}

Let $H_{n,i} = H_{n,i}^{(\l,\mu)}$ denote the norm square of $Y_{n,i}$, defined by
\begin{equation} \label{eq:norm}
  H_{n,i}:= c_{\l,\mu} \int_0^{2 \pi}  \left | \wt Y_{n,i}(\t) \right |^2 \left[ h_{\l,\mu}^{(2 k)}(\cos \t, \sin\t) \right]^2 \dd\t,
\end{equation}
where $c_{\l,\mu}$ is chosen so that $c_{\l,\mu} \int_0^{2 \pi}  [h_{\l,\mu}^{(2k)}(\cos \t, \sin\t)]^2 \dd\t =1$. Its value 
is independent of $k$ and, as can be easily verified,
\begin{equation} \label{eq:norm-cost}
  c_{\l,\mu} =  \frac{\Gamma(\l+\mu+1)}{2 \Gamma(\l+1) \Gamma(\mu+1)}. 
\end{equation}
Let $\xi(\t):= (\cos \t, \sin \t)$. Denote by $P_n(h_{\l,\mu}^{(2 k)};\cdot, \cdot)$ the reproducing kernel of 
$\CH_n(h^{(2k)}_{\l,\mu})$, which is uniquely determined by the property 
$$
 c_{\l,\mu} \int_0^{2\pi} P_n\left (h_{\l,\mu}^{(2 k)};\xi(\t),\xi(\phi) \right) \wt Y(\t)\left[ h_{\l,\mu}^{(2 k)}(\xi(\t)) \right]^2 \dd\t
      = \wt Y(\phi), \quad \forall Y \in \CH_n(h^{(2 k)}_{\l,\mu}).
$$
In terms of an orthogonal basis, the kernel can be written as
$$
P_n\left (h_{\l,\mu}^{(2 k)};x, y \right) =  \frac{Y_{n,1}(x)Y_{n,1}(y)}{H_{n,1}} +
     \frac{Y_{n,2}(x)Y_{n,2}(y)}{H_{n,2}}, 
$$
where $x = (x_1,x_2)$ and $y = (y_1,y_2)$. The kernel satisfies a closed form in terms of the intertwining 
operator, denoted by $V_{\l,\mu}$, and the Gegenbauer polynomial \cite{X97},
\begin{equation} \label{eq:h-zonal}
   P_n\left (h_{\l,\mu}^{(2 k)}; x, y \right) = \frac{n+k(\l + \mu)}{k(\l + \mu)} 
      V_{\l,\mu} \left[ C_n^{(\l +\mu)k} (\la \cdot, y \ra)\right](x)
\end{equation}
for $\|x\| = \|y\| =1$. The Poisson kernel of the $h$-harmonics, denoted by $P(h_{\l,\mu}^{(2 k)}; \cdot,\cdot)$ is defined 
by the property that 
$$
 c_{\l,\mu} \int_0^{2\pi} P\left (h_{\l,\mu}^{(2 k)}; \xi(\t),\xi(\phi) \right) \wt Y(\t) \dd\t = \wt Y(\phi), \quad 
    \forall Y \in \CH_n(h^{(2 k)}_{\l,\mu}),    \quad  \forall n \in \NN_0.
$$
It satisfies a closed formula for $\|x\| = \|y\|=1$ and $0 \le r < 1$,
\begin{align} \label{eq:Poisson}
 P \left (h_{\l,\mu}^{(2 k)}; x, r y \right) &\, = \sum_{n=0}^\infty P_n\left (h_{\l,\mu}^{(2 k)}; x, y\right) r^n \\
       &\, = V_{\l,\mu} \left[\frac{1-r^2}{(1- 2 r \la \cdot,y \ra + r^2)^{(\l+\mu)k+1}}\right](x). \notag
\end{align}

\subsection{The dihedral group $I_{2k+1}$}
For the Dihedral group $I_{2k+1}$, we choose the positive root system as 
$$
   v_j = \big(\sin (\tfrac{j \pi}{2k+1}), - \cos (\tfrac{j \pi}{2k+1}) \big), \qquad j=0,1,\ldots, 2k. 
$$
There is only one conjugacy class and one parameter $\l \ge 0$. The Dunkl operators in this case can 
be derived from those of $I_{2k}$ by setting $\mu = 0$ and replace $k$ by $2k+1$. All other results 
discussed in the previous subsection also hold under the same conversion. 

\subsection{The dihedral group $I_k$ with one parameter}

We are interested in the case of one parameter, that is, $\mu = \l$ when $k$ is even. With one parameter, 
our setting for $I_{2k}$ and $I_{2k+1}$ can be unified. 

For the dihedral group $I_k$ with $k= 2, 3, 4, \ldots$, the Dunkl operators with one parameter $\l$ are 
those given in \eqref{eq:Dj} with the root system given by \eqref{eq:vj} and the reflections given by \eqref{eq:sj}. 
The corresponding weight function is given by 
\begin{equation} \label{eq:h_l}
  h_{\l}^{(k)}(x_1,x_2) := r^{\l} |\sin (k\t) |^{\l}, \quad \l \ge 0, 
\end{equation}
for $k=2, 3, 4, \ldots$. Accordingly, we denote the intertwining operator by $V_\l$. In this setting, the Poisson kernel
\eqref{eq:Poisson} becomes
\begin{equation} \label{eq:Poisson2}
 P\left (h_{\l}^{(k)}; x, ry \right) = \sum_{n=0}^\infty P_n\left (h_{\l}^{(k)}; x, y \right)r^n
        = V_{\l} \left[\frac{1-r^2}{(1- 2 r \la \cdot,y \ra + r^2)^{k \l+1}}\right](x)
\end{equation}
for $\|x\| = \|y\|=1$. For $I_{2k}$, the parameter $2k  \l $ agrees with $k (\l+\mu)$ in \eqref{eq:Poisson}
when $\mu = \l$, whereas for $I_{2k+1}$ the parameter $\mu = 0$ in \eqref{eq:Poisson}. 
 
\section{The intertwining operator}
\setcounter{equation}{0}

We first give a proof of Theorem \ref{thm:main}, which we reformulate it below, using the polar coordinates 
$x_1= r \cos \t$, $x_2 = r \sin \t$. 

\begin{thm} \label{thm:main2}
Let $f$ be a differentiable function on $\RR$. For $0 \le p \le 2k-1$, define 
$$
   F_p (x_1,x_2 ):= f \left(\cos  \left(\frac{ p \pi}{k}\right)x_1 + \sin \left(\frac{ p \pi}{k}\right) x_2 \right). 
$$
Then, for $k =2,3,4,\ldots$, the intertwining operator $V_\l$ for the dihedral group $I_k$ with one parameter $\l$ satisfies 
\begin{align} \label{eq:main2}
 V_\l F_p(x_1,x_2) = a_\l^{(k)}
     \int_{T^{k-1}} f \Bigg (r \sum_{j=0}^{k-1} \cos \left(\t - \frac{p \pi}{k} - \frac{2 j\pi}{k}\right) u_j \Bigg) 
        u_0 \prod_{i=0}^{k-1} u_i^{\l-1}\dd u. 
\end{align}
\end{thm}

\begin{proof}
We only need to consider $0 \le p \le k-1$, since we can replace $f(t)$ by $f(-t)$ if $k \le p \le 2k-1$ by
$\sin (\pi +\t) = - \sin \t$ and $\cos (\pi +\t) = - \cos \t$. Let $p$ be fixed, $0 \le p \le k-1$. For simplicity, 
we shall write
$$
  \Psi(\t,u):= \sum_{j=0}^{k-1} \cos \left(\t - \frac{p \pi}{k} - \frac{2 j\pi}{k}\right) u_j.
$$
Our goal is to verify that 
\begin{equation}\label{eq:intertw}
  D_1 V_\l F_p = V_\l \partial_1 F_p \quad \hbox{and}\quad  D_2 V_\l F_p = V_\l \partial_2 F_p,
\end{equation}
where $D_1$ and $D_2$ are the Dunkl operators in \eqref{eq:Dj}. First we consider the difference 
part for each reflection. 

Fix $\ell$, $0\le \ell \le k-1$. In polar coordinates, \eqref{eq:sj} becomes  
$$
  x \s_j = r \left( \cos (\tfrac{2j \pi}{k}-\t),  \sin (\tfrac{2j \pi}{k}-\t) \right), \qquad 0 \le j \le k-1,
$$ 
so that the reflection in $v_\ell$ is simply a shift in $\t$ variable. In particular,
\begin{align*}
 V_\l F_p (x \s_\ell) = &\, a_\l^{(k)}
     \int_{T^{k-1}} f \Bigg (r \sum_{j=0}^{k-1} \cos \left(\frac{2 \ell \pi}{k}- \t - \frac{p \pi}{k} - \frac{2 j\pi}{k}\right) u_j \Bigg) 
        u_0 \prod_{i=0}^{k-1} u_i^{\l-1}\dd u \\
     = &\, a_\l^{(k)} 
     \int_{T^{k-1}} f \Bigg (r \sum_{j=0}^{k-1} \cos \left(\t - \frac{2 (\ell -j) \pi}{k}+ \frac{p \pi}{k}\right) u_j \Bigg) 
        u_0 \prod_{i=0}^{k-1} u_i^{\l-1}\dd u.
\end{align*}
We need to consider two cases. 

\medskip\noindent
{\it Case 1.} $\ell \ge p$. Here we write the sum inside $f$ as 
\begin{align*}
 & \sum_{j=0}^{k-1} \cos \left(\t - \frac{2 (\ell -j) \pi}{k}+ \frac{p \pi}{k}\right) u_j  
     = \sum_{j=0}^{k-1} \cos \left(\t -  \frac{p \pi}{k} + \frac{2 (\ell - p - j) \pi}{k}\right) u_j \\
 &  \quad = \sum_{j=0}^{\ell -p} \cos \left(\t - \frac{p \pi}{k} - \frac{2 j\pi}{k}\right) u_{\ell-p-j} 
     +  \sum_{j=\ell-p+1}^{k-1} \cos \left(\t -  \frac{p \pi}{k} - \frac{2 j\pi}{k}\right) u_{k+\ell-p-j}.
\end{align*}
Making a change of variables $u_j = u_{\ell -p-j}$ for $0 \le j \le \ell-p$ and $u_j = u_{k+\ell-p-j}$ for $\ell-p+1\le j \le k-1$,
we see that 
$$
 V_\l F_p (x \s_\ell) =  a_\l^{(k)} 
     \int_{T^{k-1}} f \big(r \Psi(\t,u) \big) u_{\ell-p} \prod_{i=0}^{k-1} u_i^{\l-1}du.
$$
Consequently, we conclude that
\begin{align*}
V_\l F_p(x)-V_\l F_p (x \s_\ell)  =   a_\l^{(k)} 
     \int_{T^{k-1}} f \big(r \Psi(\t,u) \big) (u_0 - u_{\ell-p}) \prod_{i=0}^{k-1} u_i^{\l-1}du.
\end{align*}
Evidently, $V_\l F_p(x) - V_\l F_p (x \s_p) = 0$. For $\ell > p$, we use the relation 
$$
  \frac{d}{d u_{\ell-p}}\left[ u_0 u_{\ell-p} \prod_{i=0}^{k-1} u_i^{\l-1}\right] = \l (u_0 - u_{\ell-p}) \prod_{i=0}^{k-1} u_i^{\l-1}, 
$$
which follows from $u_0=1-u_1-\cdots - u_{k-1}$, so that an integrating by parts gives
\begin{align*}
& \l \left[V_\l F_p(x)-V_\l F_p (x \s_\ell)\right]  =   -  a_\l^{(k)} 
     \int_{T^{k-1}} f'\big(r \Psi(\t,u) \big)  \\
 & \qquad \times r \left(\cos \Big(\t - \frac{(2 \ell-p)\pi}{k}\Big) - \cos  \left(\t -\frac{p \pi}{k}\right) \right) 
       u_0 u_{\ell-p} \prod_{i=0}^{k-1} u_i^{\l-1}du.
\end{align*}
By \eqref{eq:vj}, we have $\la x, v_\ell \ra = r \sin (\frac{\ell \pi}{k}- \t)$ in polar coordinates. Hence, using 
$$
\cos \Big(\t - \frac{(2 \ell-p)\pi}{k}\Big) - \cos  \left(\t -\frac{p \pi}{k}\right) 
      = 2 \sin \Big(\t - \frac{\ell \pi}{k}\Big) \sin \Big( \frac{(\ell -p)\pi}{k} \Big),
$$
we conclude that 
\begin{align*}
  \l \frac{ V_\l F_p(x)-V_\l F_p (x \s_\ell)}{\la x,v_\ell \ra}  =  2 a_\l^{(k)} 
     \int_{T^{k-1}}  f'\big(r \Psi(\t,u) \big) \sin \Big( \frac{(\ell -p)\pi}{k} \Big) u_0 u_{\ell -p} \prod_{i=0}^{k-1} u_i^{\l-1}\dd u.
\end{align*}
 
\medskip\noindent
{\it Case 2.} $\ell < p$. Here we write the sum inside $f$ as 
\begin{align*}
 & \sum_{j=0}^{k-1} \cos \left(\t - \frac{2 (\ell -j) \pi}{k}+ \frac{p \pi}{k}\right) u_j  
   = \sum_{j=0}^{k-1} \cos \left(\t -  \frac{p \pi}{k} + \frac{2 (\ell - p - j + k) \pi}{k}\right) u_j \\
 & = \sum_{j=0}^{k-p+\ell} \cos \left(\t - \frac{p \pi}{k} - \frac{2 j \pi}{k}\right) u_{k-p+\ell-j}  
    + \sum_{j=k-p+\ell+1}^{k-1} \cos \left(\t -  \frac{p \pi}{k} - \frac{2 j\pi}{k}\right) u_{2k-p+\ell-j}. 
\end{align*}
Making a change of variables $u_j = u_{k- p+\ell-j}$ for $0 \le j \le k-p+\ell$ and $u_j = 2k-p+\ell-j$ for $k-p+\ell+1\le j \le k-1$, we see that 
$$
 V_\l F_p (x \s_\ell) =  a_\l^{(k)} 
     \int_{T^{k-1}} f \big(r \Psi(\t,u) \big) u_{k-p+\ell} \prod_{i=0}^{k-1} u_i^{\l-1}\dd u.
$$
We can now follow the procedure in Case 1 to conclude that
\begin{align*}
  \l \frac{ V_\l F_p(x)-V_\l F_p (x \s_\ell)}{\la x,v_\ell \ra}  = &\, 2 a_\l^{(k)} 
     \int_{T^{k-1}} f'\big(r \Psi(\t,u) \big)
 \sin \Big( \frac{(\ell -p)\pi}{k} \Big) u_0 u_{k-p+\ell} \prod_{i=0}^{k-1} u_i^{\l-1}\dd u.
\end{align*}
\medskip

We are now ready to verify \eqref{eq:intertw}. We consider the intertwining identity for $D_1$ first. Putting the 
two cases together, we obtain 
$$
 \l \sum_{\ell=0}^{k-1}\frac{ V_\l F_p(x)-V_\l F_p (x \s_\ell)}{\la x,v_\ell \ra} \sin \left(\frac{\ell \pi}{k} \right)   
 = \, a_\l^{(k)} \int_{T^{k-1}} f'\big(r \Psi(\t,u) \big) S(u) u_0   \prod_{i=0}^{k-1} u_i^{\l-1}\dd u.
$$
where 
\begin{align*}
S(u) =  2 \sum_{\ell=0}^{p-1} \sin \left(\frac{\ell \pi}{k} \right) \sin \Big( \frac{(\ell -p)\pi}{k} \Big) u_{k-p+\ell}
    +  2 \sum_{\ell=p+1}^{k-1} \sin \left(\frac{\ell \pi}{k} \right) \sin \Big( \frac{(\ell-p)\pi}{k} \Big) u_{\ell-p}. 
\end{align*}
For $0 \le \ell \le p-1$, we use the identity 
\begin{align*}
& 2 \sin \left(\frac{\ell \pi}{k} \right) \sin \Big( \frac{(\ell -p)\pi}{k} \Big) 
  =  2 \sin \Big(\frac{(k-p+\ell) \pi}{k} + \frac{p \pi}{k} \Big) \sin \Big( \frac{(k-p+\ell)\pi}{k} \Big)  \\
& = \left(1- \cos \Big(\frac{2(k-p+\ell) \pi}{k} \Big) \right) \cos \Big(\frac{p \pi}{k} \Big) 
    +   \sin \Big( \frac{2(k-p+\ell)\pi}{k} \Big) \sin \Big(\frac{p \pi}{k} \Big), 
\end{align*}
where the index $k-p+\ell$ matches that of $u_{k-p+j}$ in the first sum of $S(u)$, whereas for $p+1\le \ell \le k-1$,
we use the identity 
\begin{align*}
& 2 \sin \left(\frac{\ell \pi}{k} \right) \sin \Big( \frac{(\ell -p)\pi}{k} \Big) 
  =  2 \sin \Big(\frac{(\ell-p) \pi}{k} + \frac{p \pi}{k} \Big) \sin \Big( \frac{(\ell-p)\pi}{k} \Big)  \\
& = \left(1- \cos \Big(\frac{2(\ell-p) \pi}{k} \Big) \right) \cos \Big(\frac{p \pi}{k} \Big) 
    +   \sin \Big( \frac{2(\ell-p)\pi}{k} \Big) \sin \Big(\frac{p \pi}{k} \Big), 
\end{align*}
where the index $\ell-p$ agrees with that of $u_{k-p}$ in the second sum of $S(u)$. Putting together we 
conclude that 
\begin{align*}
 S(u) =& \,  \cos \Big(\frac{p \pi}{k} \Big) \sum_{\ell=0}^{k-1}  \left(1- \cos \Big(\frac{2 \ell \pi}{k} \Big) \right) u_j
    + \sin \Big(\frac{p \pi}{k} \Big) \sum_{\ell=0}^{k-1}\sin \Big( \frac{2 \ell \pi}{k} \Big) \\
      = & \,   \cos \Big(\frac{p \pi}{k} \Big) -  \sum_{\ell=0}^{k-1} \cos \Big(\frac{p \pi}{k} + \frac{2 \ell \pi}{k} \Big),
\end{align*}
since $ \sum_{\ell=0}^{k-1} u_\ell =1$. Furthermore, taking derivative in \eqref{eq:main2}, we obtain 
\begin{align*}
 \frac{ \partial}{\partial x_1} V_\l F_p(x_1,x_2) =  a_\l^{(k)} \int_{T^{k-1}} f' \big(r  \Psi(\t,u)) 
      \sum_{j=0}^{k-1}  \cos \left(\frac{p \pi}{k} + \frac{2 j\pi}{k}\right) u_j  \,   u_0 \prod_{i=0}^{k-1} u_i^{\l-1}\dd u. 
\end{align*}
Hence, by the definition of $D_1$ in \eqref{eq:Dj}, we conclude that 
\begin{align*}
  D_1 V_\l F_p(x) &\, = \cos \Big(\frac{p \pi}{k} \Big)  \int_{T^{k-1}} f' \big(r  \Psi(\t,u)) u_0 \prod_{i=0}^{k-1} u_i^{\l-1}\dd u\\
    &\, = \int_{T^{k-1}} \frac{\partial}{\partial x_1} F_p \big(r  \Psi(\t,u)) u_0 \prod_{i=0}^{k-1} u_i^{\l-1}\dd u
     = V_\l \partial_1 F_p (x)
\end{align*}
since $\partial_1 F_p(x_1,x_2) = \cos (\frac{p \pi}{k}) f'(\cos (\frac{p \pi}{k})x_1 + \sin (\frac{p \pi}{k})x_2)$. This
verifies the first identity in \eqref{eq:intertw}. 

The second identity in \eqref{eq:intertw} is verified similarly. From the two cases that we consider for individual 
difference operator, we obtain
$$
 \l \sum_{\ell=0}^{k-1}\frac{ V_\l F_p(x)-V_\l F_p (x \s_\ell)}{\la x,v_\ell \ra} \cos \left(\frac{\ell \pi}{k} \right)   
 = \, a_\l^{(k)} \int_{T^{k-1}} f'\big(r \Psi(\t,u) \big) C(u) u_0   \prod_{i=0}^{k-1} u_i^{\l-1}\dd u.
$$
where 
\begin{align*}
C(u) =  2 \sum_{\ell=0}^{p-1} \cos \left(\frac{\ell \pi}{k} \right) \sin \Big( \frac{(\ell -p)\pi}{k} \Big) u_{k-p+\ell}
    +  2 \sum_{\ell=p+1}^{k-1} \cos \left(\frac{\ell \pi}{k} \right) \sin \Big( \frac{(\ell-p)\pi}{k} \Big) u_{\ell-p}. 
\end{align*}
For $0 \le \ell \le p-1$, we use the identity 
\begin{align*}
& 2 \cos \left(\frac{\ell \pi}{k} \right) \sin \Big( \frac{(\ell -p)\pi}{k} \Big)  \\
& = \cos \Big(\frac{p \pi}{k} \Big)  \sin \Big( \frac{2(k-p+\ell)\pi}{k} \Big) 
       - \sin \Big(\frac{p \pi}{k} \Big)\left(1- \cos \Big(\frac{2(k-p+\ell) \pi}{k} \Big) \right), 
\end{align*}
whereas for $p+1\le \ell \le k-1$, we use the identity 
\begin{align*}
& 2 \cos \left(\frac{\ell \pi}{k} \right) \sin \Big( \frac{(\ell -p)\pi}{k} \Big) \\
& = \cos \Big(\frac{p \pi}{k} \Big)  \sin \Big( \frac{2(\ell-p)\pi}{k} \Big) 
       - \sin \Big(\frac{p \pi}{k} \Big)\left(1- \cos \Big(\frac{2(\ell-p) \pi}{k} \Big) \right);
\end{align*}
together they lead to 
\begin{align*}
 C(u) = -\sin \Big(\frac{p \pi}{k} \Big) +  \sum_{\ell=0}^{k-1} \sin \Big(\frac{p \pi}{k} + \frac{2 \ell \pi}{k} \Big).
\end{align*}
Furthermore, taking derivative in \eqref{eq:main2}, we obtain 
\begin{align*}
 \frac{ \partial}{\partial x_2} V_\l F_p(x_1,x_2) =  a_\l^{(k)} \int_{T^{k-1}} f' \big(r  \Psi(\t,u)) 
      \sum_{j=0}^{k-1}  \sin \left(\frac{p \pi}{k} + \frac{2 j\pi}{k}\right) u_j  \,   u_0 \prod_{i=0}^{k-1} u_i^{\l-1}\dd u. 
\end{align*}
Hence, by the definition of $D_2$ in \eqref{eq:Dj}, we conclude that 
\begin{align*}
  D_2 V_\l F_p(x) &\, = \cos \Big(\frac{p \pi}{k} \Big)  \int_{T^{k-1}} f' \big(r  \Psi(\t,u)) u_0 \prod_{i=0}^{k-1} u_i^{\l-1}\dd u\\
    &\, = \int_{T^{k-1}} \frac{d}{dx_2} F_p \big(r  \Psi(\t,u)) u_0 \prod_{i=0}^{k-1} u_i^{\l-1}\dd u
     = V_\l \partial_2 F_p (x)
\end{align*}
since $\partial_2 F_p(x_1,x_2) = \sin (\frac{p \pi}{k}) f'(\cos (\frac{p \pi}{k})x_1 + \sin (\frac{p \pi}{k})x_2)$. This
verifies the second identity in \eqref{eq:intertw}. The proof is completed. 
\end{proof}

Let us mention some consequences of our main theorem. 

\begin{cor}
Let $f$ be a differentiable function on $\RR$. Then, for $k=2,3,4,\ldots$, 
\begin{align} \label{eq:Vf1}
    V f(\{\cdot\}_1) (x_1,x_2) & \, = a_\l^{(k)} \int_{T^{k-1}} f \big(c(u) x_1+s(u) x_2\big) u_0^\l \prod_{i=1}^{k-1} u_i^{\l-1}du \\
   &\, = a_\l^{(k)} \int_{T^{k-1}} f \Bigg (r \sum_{j=0}^{k-1} \cos \left(\t - \frac{2 j\pi}{k}\right) u_j \Bigg)
        u_0^\l \prod_{i=1}^{k-1} u_i^{\l-1}\dd u. \notag 
\end{align}
Furthermore, if $k$ is even, then
\begin{align} \label{eq:Vf2}
  V f(\{\cdot\}_2) (x_1,x_2) &\, = a_\l^{(k)} \int_{T^{k-1}} f \big(c(u) x_2 - s(u) x_1\big) u_0^\l \prod_{i=1}^{k-1} u_i^{\l-1}du \\
    &\, = a_\l^{(k)} \int_{T^{k-1}} f \Bigg (r \sum_{j=0}^{k-1} \sin \left(\t - \frac{2 j\pi}{k}\right) u_j \Bigg)  
        u_0^\l \prod_{i=1}^{k-1} u_i^{\l-1}\dd u. \notag 
\end{align}
\end{cor}

\begin{proof}
The first identity is \eqref{eq:main} with $p =0$. The second identity is \eqref{eq:main} with $p = m/2$, which
is an integer only if $m$ is even. 
\end{proof}

The result and the proof may seem to suggest that we only need to prove the theorem for $V_\l f(\{\cdot \}_1)$. 
However, a moment reflection shows that this is not the case since 
$V_\l [f (a \{\cdot\}_1+b \{\cdot\}_2)](x) \ne  V_\l  \left[ f (\{\cdot \}_1)\right](a x_1+b x_2)$ in general. 

The statement of the theorem and its proof may suggest a possible formula for $V_\l f(\{\cdot\}_1, \{\cdot\}_2)$. 
However, the obvious choice does not work out and a full formula is still elusive at the time of this writing.

\section{Poisson kernel of $h$-harmonics and orthogonal polynomials}
\setcounter{equation}{0}

The closed formula of $V_\l$ in the previous section has implications on $h$-harmonics and their associated
orthogonal polynomials. We start with a short subsection on the connection of $h$-harmonics and orthogonal 
polynomials. 

\subsection{$h$-harmonics and orthogonal polynomials}

Let $w$ be a nonnegative weight function on $[-1,1]$. Let $p_n(w)$ denotes the orthogonal polynomial of 
degree $n$ and 
$$
   c_\l \int_{-1}^1 p_n(w;x) p_m(w;x) w(x) \dd x = h_n(w) \delta_{n,m}, 
$$
where $c_\l$ is the normalization constant of $w_\l$. The Poisson kernel associated with $w$ is defined by
$$
  \phi_r (w; x,y) = \sum_{n=0}^\infty \frac{p_n(w; x) p_n(w_;y)}{h_n(w)} r^n, \qquad 0 \le r <1. 
$$

For the weight function $w_\l(x) = (1-x^2)^{\l-\f12}$, $\l > -\f12$, the corresponding orthogonal polynomials 
are the Gegenbauer polynomials $C_n^\l$, whcih are known to satisfy two generating functions,
\begin{equation}\label{eq:Gegen1} 
   \frac{1}{(1-2r x  +r^2)^\l} = \sum_{n=0}^\infty C_n^\l (x) r^n, \qquad 0 \le r < 1,
\end{equation}
and, for $\l > 0$, 
\begin{equation}\label{eq:Gegen2} 
   \frac{1-r^2}{(1-2r x  +r^2)^{\l+1}} = \sum_{n=0}^\infty \frac{n+\l}{\l} C_n^\l (x) r^n, \qquad 0 \le r < 1.
\end{equation}
The righthand side of the second one, \eqref{eq:Gegen2}, is in fact the Poisson kernel $\phi_r(w_\l; x,1)$. 

For the $h$-spherical harmonics associated to $h_\l$ defined in \eqref{eq:h_l}, an orthogonal basis 
$\{Y_{n,1}^{\l,k}, Y_{n,2}^{\l,k}\}$ of $\CH_n(h_{\l})$ can be given in terms of the orthogonal polynomials on $[-1,1]$ 
with respect to the weight functions $w_{\l-\f12}^{(k)}$ and $w_{\l+\f12}^{(k)}$, respectively, where 
\begin{equation} \label{eq:weightR}
  w_{\l \mp \f12}^{(k)} (t) =  | U_{k-1}(t)|^{2\l} (1-t^2)^{\l \mp \f12}, \qquad -1 < t < 1.
\end{equation}
The normalization constant $b_{\l\mp \f12} = 1/ \int_{-1}^1 w_{\l \mp \f12}^{(k)} (t)dt$ is easily seen to be
$$
  b_{\l  - \f12} = 2 c_\l = \frac{\Gamma(\l+\f12)^2}{\Gamma(2\l+1)} \quad \hbox{and}\quad 
  b_{\l + \f12} = 2 b_{\l - \f12} = \frac{2 \Gamma(\l+\f12)^2}{\Gamma(2\l+1)}. 
$$

\begin{prop} \label{prop:Yn-pn}
For $k=2,3,4,\ldots$ and $\l \ge 0$, 
\begin{align} \label{eq:Yn-pn}
\begin{split}
   Y_{n,1}^{\l,k} (x_1,x_2) &\, =  r^n p_n\left(w_{\l-\f12}^{(k)}; x_1\right), \qquad n =0,1,2,\ldots,\\
   Y_{n,2}^{\l,k} (x_1,x_2)&\, =  r^n x_2 p_{n-1} \left(w_{\l+\f12}^{(k)};  x_1\right), \qquad n =1,2,\ldots.
\end{split}
\end{align}
Furthermore, the norm $H_{n,i}^{\l,k}$ of $Y_{n,i}^{\l,k}$ satisfies
\begin{align} \label{eq:Yn-pn-norm}
  H_{n,1}^{\l,k} = h_n\left(w_{\l-\f12}^{(k)}\right) \quad \hbox{and} \quad  
  H_{n,2}^{\l,k} = \frac12 h_n\left(w_{\l+\f12}^{(k)}\right). 
\end{align}
\end{prop}

\begin{proof}
The relation \eqref{eq:Yn-pn} is known in a much more general setting, see \cite[Section 4.2]{DX}. We compute
the norm of $Y_{n,2}^{\l,\k}$, 
\begin{align*}
   H_{n,2}^{\l,k} & = c_\l \int_{-\pi}^\pi \left|Y_{n,2}^{\l,k}(\cos \t,\sin\t)  \right|^2 |\sin (k \t)|^{2\l}  \dd \t  \\
    &  = 2 c_\l \int_{0}^\pi \left |p_{n-1} \left(w_{\l+\f12}^{(k)};  \cos \t\right)\right|^2 (\sin \t)^2 |\sin (k \t)|^{2\l} \dd \t  \\
    &  = b_{\l - \f12} \int_{-1}^1 \left |p_{n-1} \left(w_{\l+\f12}^{(k)};  t\right)\right|^2 w_{\l+\f12}^{(k)} (t) \dd \t  
     = \frac12 h_n (w_{\l+\f12}),
\end{align*}
where we have used $b_{\l+\f12} = 2 b_{\l-\f12}$ in the last step. The case of $H_{n,1}^{\l}$ can be verified
similarly and is easier. 
\end{proof}

The polynomials $p_n(w_{\l \mp \f12}^{(k)};\cdot)$ are called {\it sieved} Gegenbauer polynomials since their 
three-term relation possesses a structure that can be viewed as a sieve is operated on the recurrence 
relations of the Gegenbauer polynomials. These polynomials are studied in \cite{AAA}, where they are 
defined by their recurrence relations.  

An explicit formula of $p_n(w_{\l \mp \f12}^{(k)};\cdot)$ is given in Proposition \ref{prop:basis}. Indeed, it is 
easy to see that $\wh Y_{mk +j}$ in \eqref{eq:basis} is a polynomial of degree $n$ in $t = \cos \t$ and 
$\wh Y_{mk +j}$ in \eqref{eq:basis} is equal to $\sin \t$ multiple of a polynomial of degree $n-1$ in $t= \cos \t$.
It turns out that a simpler basis can be given in this case; see Proposition \ref{prop:GGk} in the following section. 
\subsection{Poisson kernels for $h$-harmonics and sieved Gegenbauer polynomials}
The formula of the intertwining operator in Theorem \ref{thm:main} can be used to derive a closed form 
formula for the Poisson kernels \eqref{eq:Poisson} when one argument is at the vertexes of a regular 
polygon. We need two lemmas. 

\begin{lem} \label{lem:integral}
Let $k =2,3,\ldots$ and $\boldsymbol{\large {\lambda}} = (\l_0,\ldots, \l_{k-1})$ with $\l_i > 0$, $0 \le i \le k-1$. 
For $(x_0,x_1,\ldots, x_{k-1}) \in \RR^k$ and $r \ge 0$ such that $r|x_i| < 1$, $0 \le i\le k-1$,  
\begin{equation*}
    \prod_{i=0}^{k-1}  \frac{ 1}{(1-  2 r x_i+r^2)^{\l_i}} =  \frac{\Gamma(|\boldsymbol{\large {\lambda}}|)}{\prod_{i=0}^{k-1}    \Gamma(\l_i)} \int_{T^{k-1}} \frac{1}{ (1- 2 r \sum_{i=0}^{k-1} x_i u_i + r^2)^{|\boldsymbol{\large {\lambda}}|}} 
      \prod_{i=0}^{k-1} u_i^{\l_i-1} \dd u.
\end{equation*}
\end{lem}

This lemma is established recently in \cite{X15}. The second lemma is elementary; a proof is outlined for completeness. 

\begin{lem} \label{lem:trig-identity}
For $k=2,3,4,\ldots$,
$$
   1 -2 r^k  \cos (k \t) + r^{2k} =  \prod_{j=0}^{k-1} \left(1 -2 r \cos \left(\t - \frac{2j \pi}{k}\right) + r^2\right). 
$$
\end{lem} 

\begin{proof}
Using $z^k -1 = \prod_{i=0}^{k-1} (z- e^{\frac{2 \pi j \i}{k}})$, it is easy to see that 
$$
  1- r^k e^{\i \t} = \prod_{j=0}^{k-1} \left(1-r e^{\i \t - \frac{2 \pi j \i}{k}} \right). 
$$
The stated formula then follows from $(1- r e^{\i \t})(1- r e^{-\i \t}) = 1-2r \cos \t +r^2$. 
\end{proof}

Our first result gives a closed formula for the Poisson kernel \eqref{eq:Poisson2} associated to the dihedral group 
$I_k$ with $k=2,3,4,\ldots$. 

\begin{thm}\label{thm:generating}
For $k= 2,3, \ldots$ and $p =0,1,\ldots, k-1$, let $y_{p,k} = (\cos \frac{p \pi}{k}, \sin \frac{p \pi}{k})$. Then, for 
$\|x\| =1$ and $0 \le r <1$, 
\begin{align} \label{eq:generating}
  &  P\left (h_{\l}; x, r y_{p,k} \right)  = \sum_{n=0}^\infty \left(\frac{Y_{n,1}(x)Y_{n,1}(y_{p,k})}{H_{k,1}} +
     \frac{Y_{n,2}(x)Y_{n,2}(y_{p,k})}{H_{n,2}} \right) r^n\\
   & \qquad \, =   \frac{1-r^2}{(1-2r (\cos \left( \frac{p \pi}{k}\right) x_1+ \sin \left( \frac{p \pi}{k}\right)x_2)+ r^2) 
            (1-2  (-1)^p r^k T_k (x_1) + r^{2k})^{\l} }.
       \notag
\end{align}
\end{thm} 

\begin{proof}
Applying \eqref{eq:main} to the function $f(t) = (1-2 r t + r^2)^{-(k\l+1)}$, the Poisson kernel in \eqref{eq:Poisson2}
becomes
\begin{align*}
    P\left (h_{\l,\mu}; x, r y_{p,k} \right)  = a_\l^{(k)} \int_{T^{k-1}}  & 
      \frac{1-r^2}{(1- 2 r \sum_{j=0}^{k-1} \cos \left(\t - \frac{2 j\pi}{k}- \frac{p \pi}{k}\right) u_j +    r^2)^{k \l+1}} \\
      &   \times  u_0^\l \prod_{i=1}^{k-1} u_i^{\l-1}\dd u. 
\end{align*}
Applying the identity in Lemma \ref{lem:integral}, then the identity in Lemma \ref{lem:trig-identity} with $\t$ replaced 
by $\t-  \frac{ p \pi}{k}$, we see that 
the above integral is equal to 
\begin{align*}
 & \frac{1-r^2}{(1-2r \cos \left(\t - \frac{p \pi}{k}\right)  + r^2) \prod_{j=0}^{k-1} (1-2 r \cos \left(\t - \frac{2 j\pi}{k}- \frac{p \pi}{k}\right)  + r^2)^{\l} } \\
  & \qquad \quad =  \frac{1-r^2}{(1-2r \cos \left(\t - \frac{p \pi}{k}\right)  + r^2)   (1-2 r^k \cos \left( k \t - p \pi \right)  + r^{2k})^{\l}},
\end{align*}
which leads to the stated result, since $\cos \left( k \t - p \pi \right) = (-1)^p \cos (k\t)$, after setting $x_1 = \cos\t$ 
and $x_2 = \sin \t$.  
\end{proof}

The Poisson kernel associated to $I_k$ can be written in terms of the kernel associated to $I_2$ and the
latter has an integral expression \cite[Theorem 7.6.11]{DX}. This is used to derive a complicated integral 
formula for the kernel in \cite{Bie}. It is worth mentioning that there has also been attempt on explicit expression 
for the kernel $V [e^{\i\la \cdot,y\ra}](x)$ in the dihedral group setting \cite{DDY}, but the result is in series rather 
than in integral. Our partial closed form of the intertwining operator gives satisfactory formulas when $y = y_{p,k}$
in both cases.

The definition of the Poisson kernel $P\left (h_{\l}; x, y \right)$ is independent of the choice of bases of 
$\CH_n(h_\l)$. For the basis given in Proposition \ref{prop:Yn-pn}, $Y_{n,1}^{\l,\k}$ is a function of $x_1$ only
and $Y_{n,2}^{\l,\k}$ contains a single $x_2$. As a consequence, we can separate the series for 
$Y_{n,1}^{\l,\k}$ and $Y_{n,2}^{\l,\k}$ by considering either the addition or the difference of 
$P\left (h_{\l}; (x_1,x_2), y_{p,m} \right)$ and $P\left (h_{\l}; (x_1,-x_2), y_{p,m} \right)$.
Given the relation \eqref{eq:Yn-pn}, this leads to the Poisson kernes for $w_{\l \pm \f12}^{(k)}$.  

\begin{thm} \label{thm:poisson2}
For $k=1, 2,3 \ldots$, $0\le p \le k-1$, $\|x\| =1$ and $0 \le r <1$, 
\begin{align} \label{eq:poisson1}
   &\phi_r  \left(w_{\l-\f12}^{(k)}; t, \cos \left(\tfrac{p \pi}{k}\right)\right) =  \sum_{n=0}^\infty 
      \frac{p_{n}\big(w_{\l-\f12}^{(k)}; t\big)p_{n}\big(w_{\l-\f12}^{(k)};  \cos \f{p \pi}{k}\big)}
          {h_n \big(w_{\l-\f12}^{(k)}\big)} r^n \\
        = &  
        \frac{(1-r^2)(1-2r \cos (\frac{p \pi}{k}) t+r^2)}{\big ( (1-2r \cos (\frac{p \pi}{k})t + r^2)^2 - 
           4r^2 \sin^2(\frac{p \pi}{k}) (1-t^2) \big) (1-2(-1)^p r^k T_k (t) + r^{2k})^{\l} }, \notag
\end{align}
and 
\begin{align} \label{eq:poisson2}
 & \phi_r \left(w_{\l+\f12}^{(k)}; t, \cos \left(\tfrac{p \pi}{k}\right)\right) =  \sum_{n=0}^\infty 
      \frac{p_{n}\big(w_{\l+\f12}^{(k)}; t\big)p_{n}\big(w_{\l+\f12}^{(k)};  \cos \f{p \pi}{k}\big)}
          {h_n \big(w_{\l+\f12}^{(k)}\big)} r^n \\
        = &  
        \frac{ 1-r^2}{ \big ( (1-2r \cos (\frac{p \pi}{k}) t + r^2)^2 - (2r \sin(\frac{p \pi}{k}))^2 (1-t^2) \big)
              \big(1-2 (-1)^p r^k T_k (t) + r^{2k} \big )^{\l} }. \notag
\end{align}
\end{thm}

\begin{proof}
The first identity \eqref{eq:poisson1} follows precisely from \eqref{eq:generating} since  
$$
\phi_r  \left(w_{\l-\f12}^{(k)}; t, \cos \left(\tfrac{p \pi}{k}\right)\right) = 
     \frac12 \left[ P\left (h_{\l,\mu}; (x_1,x_2), r y_{p,m} \right) + P\left (h_{\l,\mu}; (x_1,-x_2), r y_{p,m} \right) \right]
$$ 
when we use \eqref{eq:Yn-pn} and \eqref{eq:Yn-pn-norm}. The second identity \eqref{eq:poisson2} is a bit more
complicated, since 
\begin{align*}
2 x_2  & \sin (\tfrac{p\pi}{k})  \phi_r  \left(w_{\l-\f12}^{(k)}; t, \cos \left(\tfrac{p \pi}{k}\right)\right) \\
 &  =   \frac12 \left[ P\left (h_{\l,\mu}; (x_1,x_2), ry_{p,m} \right) - P\left (h_{\l,\mu}; (x_1,-x_2), ry_{p,m} \right) \right]
\end{align*}
by \eqref{eq:Yn-pn} and \eqref{eq:Yn-pn-norm}, where the factor $2$ comes from \eqref{eq:Yn-pn-norm}. 
Working out the righthand side by \eqref{eq:generating} and canceling $2 x_2 \sin (\tfrac{p\pi}{k})$ that appears 
in the nominator, we obtain \eqref{eq:poisson2}. 
\end{proof}

The statement of the above theorem includes the case $k=1$, for which \eqref{eq:poisson1} is the classical 
identity \eqref{eq:Gegen2} for the Gegenbauer polynomials. For $k > 1$, these results are new except in the 
case of $p=0$ for $w_{\l -\f12}$. We state the case $p=0$ as a corollary. 

\begin{cor} \label{cor:generating}
For $k=1, 2,3 \ldots$, $t \in [-1,1]$ and $0 \le r <1$, 
\begin{align} \label{eq:generating1}
   \sum_{n=0}^\infty \frac{p_{n}\big(w_{\l-\f12}^{(k)}; t\big)p_{n}\big(w_{\l-\f12}^{(k)};  1\big)}
          {h_n \big(w_{\l-\f12}^{(k)}\big)} r^n = \frac{1-r^2}{(1-2r t + r^2)(1-2  r^k T_k (t) + r^{2k})^{\l} }, 
\end{align}
and 
\begin{align} \label{eq:generating2}
  \sum_{n=0}^\infty \frac{p_{n}\big(w_{\l+\f12}^{(k)}; t\big)p_{n}\big(w_{\l+\f12}^{(k)};  1\big)}
          {h_n \big(w_{\l+\f12}^{(k)}\big)} r^n =   
        \frac{ 1-r^2}{(1-2r t + r^2)^2 (1-2 r^k T_k (t) + r^{2k})^{\l} }.  
\end{align}
\end{cor}

The identity \eqref{eq:generating1} appeared in \cite{AAA} for the basis defined recursively therein; it
was treated as a generating function of $p_{n}(w_{\l-\f12}^{(k)}; t)$ but was not identified as the 
Poisson kernel. The identity \eqref{eq:generating2} is new; a different generating function is given for 
$p_{n}(w_{\l+\f12}^{(k)}; t)$ in \cite{AAA}, see \eqref{eq:generatingAAA} below. 

\section{Sieved Gegenbauer polynomials}
\setcounter{equation}{0}

Using the connection between $h$-harmonics and orthogonal polynomials in Proposition \ref{prop:Yn-pn},
an explicit basis of orthogonal polynomials with respect to $w_{\l-\f12}$ can be derived from 
Proposition \ref{prop:basis}. An simpler form of the basis can be derived from the relation \eqref{eq:generating1}.

\begin{prop} \label{prop:GGk}
For $k = 2, 3, \ldots$, $n= m k+ j$ with $j =0,1,\ldots, k-1$, a basis of orthogonal polynomial with
respect to $w_{\l-\f12}^{(k)}$ is given by
\begin{align}\label{eq:basis-a-1B}
  p_{k m+j}\big (w_{\l -\f12}^{(k)};\cos \t \big) = \cos(j\t) C_m^{\l+1}(\cos(k\t)) - \cos ((k-j)\t)C_{m-1}^{\l+1}(\cos(k\t)).
\end{align}
Moreover, the $L^2$ norm of these polynomials are given by
\begin{align}\label{eq:basis-a-1normA}
   h_{km} \big (w_{\l -\f12}^{(k)} \big ) = \frac{m + 2 \l}{2(m+\l)} p_{k m}(w_{\l -\f12}^{(k)};1)  
        =  \frac{m + 2 \l}{m+\l} \frac{(2\l+1)_m}{2 \, m!}.
\end{align}
and, for $1\le j \le k-1$,
\begin{align}\label{eq:basis-a-1normB}
    h_{km+j} \big(w_{\l -\f12}^{(k)} \big) = &\, \frac{1}{2}  p_{k m+j}(w_{\l -\f12}^{(k)};1) =  \frac{(2\l+1)_m}{2 \, m!}.
\end{align}
\end{prop}
 
\begin{proof}
Using \eqref{eq:Gegen1}, we can verify directly, or with the help of a computer algebra system, that
\begin{align*}
&  2 \sum_{j=1}^{k-1} \sum_{m=0}^\infty p_{k m+j}(w_{\l -\f12}^{(k)};\cos \t) r^{km +j }\\
  &  =  2 \sum_{m=0}^\infty C_m^{\l+1} (\cos (k \t)) r^{k m} \left( \sum_{j=0}^{k-1} \cos(j\t) r^j 
      -\sum_{j=1}^{k-1} \cos ((k-j)\t) r^{k+j} \right) \\
  & =  \frac{1}{(1+ r^k \cos (k\t) + r^{2k})^{\l+1}} \left( \frac{(1 - r^2) (1 - 2 r^k \cos(k \t) + r^{2k})}{1 - 2 r \cos \t + r^2}
       - (1-r^{2k}) \right). 
\end{align*}
For $k=0$, we use the identity  
\begin{equation} \label{eq:Gegen-recur}
   C_m^{\l+1}(t) - t C_{m-1}^{\l+1}(t) =  \frac{m + 2 \l}{2 \l} C_m^\l(t),
\end{equation}
which can be easily verified by using the ${}_2F_1$ expansion of these polynomials, so that 
\begin{align*}
 \sum_{m=0}^\infty  \frac{2(m + \l)}{m+2 \l}  p_{k m}(w_{\l -\f12}^{(k)};\cos \t) r^{km} &\, =
    \sum_{m=0}^\infty  \frac{ m + \l}{\l}  C_m^\l(\cos k \t) r^{km} \\
       & \, = \frac{1-r^{2k}}{ (1-2  r^k \cos(k \t) + r^{2k})^{\l+1} }
\end{align*}
by \eqref{eq:Gegen2}. Together, we conclude that 
\begin{align*}
   \sum_{m=0}^\infty  \frac{2(m + \l)}{m+2 \l}  p_{k m}(w_{\l -\f12}^{(k)};\cos \t) r^{km} & +
    2 \sum_{j=1}^{k-1} \sum_{m=0}^\infty p_{k m+j}(w_{\l -\f12}^{(k)};\cos \t)r^{km+j} \\
  &     = \frac{1-r^2}{(1-2r \cos \t + r^2)(1-2  r^k \cos(k \t) + r^{2k})^{\l} }.
\end{align*}
Consequently, by \eqref{eq:generating1}, we see that 
$$
 \frac{p_{km}\big(w_{\l-\f12}^{(k)};  1\big)} {h_{km} \big(w_{\l-\f12}^{(k)}\big)} = \frac{2(m + \l)}{m+2 \l} \quad \hbox{and}\quad
 \frac{p_{km+j}\big(w_{\l-\f12}^{(k)};  1\big)} {h_{km+j} \big(w_{\l-\f12}^{(k)}\big)} = 2, \quad j \ge 1,
$$
from which the norm $h_n\big(w_{\l-\f12}^{(k)}\big)$ can be deduced from $C_n^\l(1) = \frac{(2\l)_n}{n!}$. This 
completes the proof. 
\end{proof}

The polynomial $p_n(w_{\l-\f12};t)$ in \eqref{prop:GGk} differs from the orthogonal polynomial $c_n^\l(x;k)$ 
given in \cite{AAA} by a multiple constant $\a_n$ as can be seen from \eqref{eq:generating1} and 
\cite[(4.10)]{AAA}.  In particular, comparing \cite[(4.13)]{AAA} and our expression for $p_{m k +j}$ leads to
the following proposition, which can also be verified directly. 

\begin{prop}
For $1 \le j \le m$ and $m=2,3,\ldots$, 
\begin{align}
  &  \sum_{\ell=0}^m T_{\ell k +j} (\cos\t) C_{m- \ell}^\l (\cos k\t) \\
   & \qquad  = \cos(j\t) C_m^{\l+1}(\cos(k\t)) - \cos ((k-j)\t)C_{m-1}^{\l+1}(\cos(k\t)). \notag
\end{align}
\end{prop}

We next consider orthogonal polynomials with respect to $w_{\l+\f12}$. These polynomials are shown in \cite{AAA}
to satisfy a generating function (\cite[(4.3)]{AAA}), 
\begin{equation} \label{eq:generatingAAA}
  \sum_{n=0}^\infty p_n\big(w_{\l + \f12};t \big) r^n =
       \frac{1}{(1-2 \cos \t r + r^2) (1- 2 T_k(\cos \t) r^k + r^{2k})^\l }.
\end{equation}
It turns out that they can be given by analogues of those in Proposition \ref{prop:GGk}, replacing 
some of the cosines by sines. We shall define  $U_{-1}(t) = 0$.

\begin{prop} \label{prop:GGk2}
For $k = 2, 3, \ldots$, $n= m k+ j$ with $j =0,1,\ldots, k-1$, a basis of orthogonal polynomial with
respect to $w_{\l+\f12}^{(k)}$ is given by, for  $0 \le j \le k-1$,
\begin{align}\label{eq:basis-a+1A}
  p_{k m+j}(w_{\l +\f12}^{(k)};\cos \t) = U_j(\cos\t) C_m^{\l+1}(\cos(k\t)) + U_{k-j-2}(\cos \t)C_{m-1}^{\l+1}(\cos(k\t)).
\end{align}
Moreover, the $L^2$ norm of these polynomials are given by, for $0 \le j \le k-2$,
\begin{align}\label{eq:basis-a+1normA}
    h_{km+j}(w_{\l +\f12}^{(k)})= \, \frac{(2\l+1)_m}{2 m!},
\end{align}
and, for $j = k-1$, 
\begin{align}\label{eq:basis-a+1normB}
    h_{km+k-1}(w_{\l +\f12}^{(k)}) = \, \frac{(2\l+1)_m}{2 m!} \frac{2\l+m+1}{\l + m + 1}.
\end{align}
\end{prop}

\begin{proof}
Just as in the previous proof, we can easily verify that 
$$
 \sum_{j=0}^{k-1} \left(U_{j-1}(\cos \t) r^{j-1} +  U_{k-j-1}(\cos \t) r^{k+j-1} \right)
      = \frac{1-2 r^k \cos (k \t) + r^{2k}}{1-2 r \cos \t + r^2}. 
$$
Together with \eqref{eq:Gegen1} applied to $C_n^{\l+1}(\cos (k\t))$,  we can then verify that the polynomials
$p_{k m+j}(w_{\l +\f12}^{(k)})$ defined in \eqref{eq:basis-a+1A} satisfy the generating function 
\eqref{eq:generatingAAA}. The $L^2$ normal of these polynomials are given in \cite[(3.4)]{AAA}. 
\end{proof}

In particular, comparing \cite[(4.4)]{AAA} and our expression for $p_{m k +j}$ leads to
the following proposition, which can also be verified directly. 

\begin{prop}
For $1 \le j \le m$ and $m=2,3,\ldots$, 
\begin{align}
  &  \sum_{\ell=0}^m U_{\ell k +j} (\cos\t) C_{m- \ell}^\l (\cos k\t) \\
   & \qquad  = U_j(\cos\t) C_m^{\l+1}(\cos(k\t)) - U_{k-j-2}(\cos\t)C_{m-1}^{\l+1}(\cos(k\t)). \notag
\end{align}
\end{prop}

\section{Two related families of orthogonal polynomials}
\setcounter{equation}{0}

In the case when $k$ is an even integer, we can relate orthogonal polynomials with respect to $w_{\l+\f12}$ 
to another set of orthogonal polynomials associated to the weight function
\begin{equation} \label{eq:wlk1}
  w_{\l-\f12}^{(k),1}(t):= (1-t) w_{\l-\f12}^{(k)}(t) = (1-t) (1-t^2)^{\l-\f12}|U_{k-1}(t)|^{2\l}.
\end{equation}
The latter can also be related to orthogonal polynomials associate to the weight function 
\begin{equation} \label{eq:wlk-1}
  w_{\l-\f12}^{(k), -1}(t):= (1+ t) w_{\l-\f12}^{(k)}(t) = (1+t) (1-t^2)^{\l-\f12}|U_{k-1}(t)|^{2\l}.
\end{equation}

\begin{prop}\label{prop:OPk,1}
For $k = 1,2,3,\ldots$, $n=0,1,2,\ldots$, 
\begin{align} \label{eq:OPk,1}
\begin{split}
  p_{2n}\big(w_{\l-\f12}^{(2k)}; \cos \t \big) &\, = p_{n}\big(w_{\l-\f12}^{(k)}; \cos (2\t) \big), \\
  p_{2n+1}\big(w_{\l-\f12}^{(2k)}; \cos\t \big) &\, = \cos \t \,  p_{n}\big(w_{\l-\f12}^{(k),1}; \cos (2\t) \big). 
\end{split}
\end{align}
Furthermore, the norms of these polynomials in their respective $L^2$ space satisfy
\begin{equation} \label{eq:OPk,1norm}
   h_{2n}\big(w_{\l-\f12}^{(2k)}\big) = h_{n}\big(w_{\l-\f12}^{(k)}\big) \quad \hbox{and}\quad 
   h_{2n+1}\big(w_{\l-\f12}^{(2k)}\big) = \f12 h_{n}\big(w_{\l-\f12}^{(k),1}\big).
\end{equation}
\end{prop}

\begin{proof}
Since the weight function $w_{\l-\f12}^{(2k)}(t)$ is even,  $p_{2n}\big(w_{\l-\f12}^{(2k)}\big)$ is even in $t$, 
so that $p_{2n}\big(w_{\l-\f12}^{(2k)}; t\big) = q_{n,1}(2 t^2 -1)$, where $q_{n,1}$ is a polynomial of degree $n$, 
and $p_{2n+1}\big(w_{\l-\f12}^{(2k)}\big)$ is odd, so that $p_{2n+1}\big(w_{\l-\f12}^{(2k)};t\big) = t q_{n,2}(2 t^2-1)$,
where $q_{n,2}$ is also a polynomial of degree $n$. Using $\cos^2\t = (1 + \cos 2 \t)/2$, we obtain 
\begin{align*}
 b_{\l-\f12} & \int_{0}^\pi p_{2n+1}\big(w_{\l-\f12}^{(2k)}; \cos \t\big) p_{2m+1}\big(w_{\l-\f12}^{(2k)}; \cos\t\big) 
   |\sin (2k \t)|^{2 \l} \dd \t \\
&  =  b_{\l-\f12} \int_{0}^\pi   (\cos \t)^2 q_{n,2}(\cos 2 \t) q_{m,2}(\cos 2 \t)  |\sin (2k \t)|^{2 \l} \dd \t  \\
&  =  b_{\l-\f12} \int_{0}^{2\pi}  \frac{1+\cos \t}{2} q_{n,2}(\cos \t) q_{m,2}(\cos \t)  |\sin (k \t)|^{2 \l} \frac{\dd \t}{2}  \\
& =   \frac{b_{\l-\f12}}{2} \int_{0}^{\pi} q_{n,2}(\cos \t) q_{m,2}(\cos \t) (1+ \cos \t)  |\sin (k \t)|^{2 \l} \dd \t,
\end{align*}
where we have changed $\t \to 2 \pi - \t$ in the integral over $[\pi, 2 \pi]$. This establishes the orthogonality of
$q_{n,2}$ with resect to the weight function $w_{\l-\f12}^{(k),1}$ as well as the relation on $L^2$ norm, since the
normalization constant for $w_{\l-\f12}^{(k),1}$ is the same as the one for $w_{\l-\f12}^{(k)}$. The verification for 
$q_{n,1}$ is similar and easier. 
\end{proof}

The relation \eqref{eq:OPk,1} allows us to derive an orthogonal basis for $w_{\l-\f12}^{(k),1}$, from which an
orthogonal basis for $w_{\l-\f12}^{(k),-1}$ also follows. 

\begin{prop} \label{prop:OPk,1basis}
For $k=1,2,3,\ldots$, $m=0, 1,2,\ldots$ and $j =0,1,\ldots, k-1$,
\begin{align} \label{eq:OPk,1basis}
  p_{k m +j} & \big(w_{\l-\f12}^{(k),1}; \cos\t \big)= \\
    &  \frac{\cos ((j+\f12) \t)}{\cos (\frac{\t}2)} C_m^{\l+1}(\cos(k \t)) -
        \frac{\cos ((k-j+\f12) \t)}{\cos (\frac{\t}2)} C_{m-1}^{\l+1}(\cos(k \t)) \notag
\end{align}
and 
\begin{align}  \label{eq:OPk,-1basis}
  p_{k m +j} &\big(w_{\l-\f12}^{(k),-1}; \cos\t \big) = (-1)^{k m+j} \\
    & \times \left[ \frac{\sin((j+\f12) \t)}{\sin(\frac{\t}2)} C_m^{\l+1}(\cos(k \t)) + 
        \frac{\sin ((k-j+\f12) \t)}{\sin (\frac{\t}2)} C_{m-1}^{\l+1}(\cos(k \t)) \right]. \notag
\end{align}
Furthermore, $h_n\big(w_\l^{(k),1} \big) = h_n\big(w_\l^{(k),-1} \big)$ for all $n=0,1,2,\ldots$. 
\end{prop}

\begin{proof}
By \eqref{eq:basis-a-1B} with $k$ replaced by $2k$, we have 
\begin{align*}
 p_{2k m+2j+1}\big (w_{\l -\f12}^{(2k)};\cos \t \big) = &\, \cos((2j+1)\t) C_m^{\l+1}(\cos(2k\t)) \\
         &   - \cos ((2k-2j-1)\t)C_{m-1}^{\l+1}(\cos(2k\t)),
\end{align*}
from which \eqref{eq:OPk,-1basis} follows from the second identity in \eqref{eq:OPk,1} by 
setting $\t \to \t/2$. 

To establish \eqref{eq:OPk,-1basis}, we observe that $p_{n}\big(w_{\l-\f12}^{(k),1}; - t \big)$ is an 
orthogonal polynomial with respect to the weight function $w_{\l-\f12}^{(k), -1}$, as can be seen
by changing variable $t \to -t$ in the orthogonality relation.  Hence, changing variable $\t \to \pi -\t$
in \eqref{eq:OPk,1basis} gives \eqref{eq:OPk,-1basis}. 
\end{proof}

We note that the righthand sides of  \eqref{eq:OPk,1basis} and  \eqref{eq:OPk,-1basis} can also 
be written in terms of the Jacobi polynomials by using
\begin{align*}
    \frac{\cos ((j+\f12) \t)}{\cos (\frac{\t}2)}&  \,= \frac{2^{2j} j!^2}{(2j)!} P_j^{(-\f12,\f12)}(\cos\t), \\ 
    \frac{\sin ((j+\f12) \t)}{\sin (\frac{\t}2)} & \,= \frac{2^{2j} j!^2}{(2j)!} P_j^{(\f12,- \f12)}(\cos\t). 
\end{align*}

We now derive a closed formula for the Poisson kernel $\phi_r\big(w_{\l-\f12}^{(k),1}\big)$. 

\begin{thm} \label{thm:poisson3}
For $k=2,3,4\ldots$, $0\le p \le k-1$ and $0 \le r <1$, 
\begin{align*}
  \phi_r & \left(w_{\l-\f12}^{(k),1}; t, \cos \left(\tfrac{2 p \pi}{k}\right)\right) =  \sum_{n=0}^\infty 
      \frac{p_{n}\big(w_{\l-\f12}^{(k),1}; t\big)p_{n}\big(w_{\l-\f12}^{(k),1};  \cos \f{2 p \pi}{k}\big)}
          {h_n \big(w_{\l-\f12}^{(k),1}\big)} r^n \\
    = &  \,    \frac{(1-r) \big( (1+r)^2 - 2 r (\cos (\frac{p \pi}{k}) + \cos \t) \big)
    }{ \big( (1+r^2)(1- 4 r \cos (\frac{p \pi}{k}) t +r^2) +4 r^2(t^2 - \sin^2(\frac{p \pi}{k}) )\big)
      \big(1-2(-1)^p r^k T_k (t) + r^{2k}\big)^{\l}} .
\end{align*}
\end{thm}

\begin{proof}
From the definition of Poisson kernels, we obtain from \eqref{eq:OPk,1} and \eqref{eq:OPk,1norm} that 
\begin{align*}
  \phi_r\big(w_{\l-\f12}^{(2k)}; \cos \t,\cos \phi\big) = \, & \phi_{r^2} \big (w_{\l-\f12}^{(k)}; \cos(2 \t), \cos(2 \phi)\big) \\
     & +  2 r \cos \t \cos \phi \, \phi_{r^2} \big (w_{\l-\f12}^{(k),1}; \cos(2 \t), \cos(2 \phi)\big).
\end{align*}
The identity \eqref{eq:poisson1} allows us to derive a closed form for 
$$
\phi_r\big(w_{\l-\f12}^{(2k)}; \cos \t,\cos (\tfrac{p \pi}{k})\big) - 
      \phi_{r^2} \big (w_{\l-\f12}^{(k)}; \cos(2 \t), \cos (\tfrac{2 p \pi}{k})\big)
$$
as a ratio, which has the denominator as a product 
\begin{align*}
& \big((1+r^4)(1- 4 r^2 \cos (\tfrac{p \pi}{k}) \cos (2\t)+r^4) +2 r^4 \big( \cos (\tfrac{2 p \pi}{k}) + \cos (4 \t)\big) \big) \\
& \quad \times \big(1-2(-1)^p r^{2k} \cos (2k \t) + r^{4k}\big)^{\l}
\end{align*}
and the nominator as 
$$
 2 r \cos \t \cos (\tfrac{p \pi}{k}) (1-r^2) \big((1+r^2)^2 - 2 r^2 (\cos (\tfrac{p \pi}{k})+ \cos (2 \t)) \big).
$$
Removing $ 2 r \cos \t \cos (\tfrac{p \pi}{k})$ in the nominator and changing $r^2$ to $r$ and $2 \t$ to $\t$,
we obtain a closed expression for $\phi_r\big(w_{\l-\f12}^{(2k),1}; \cos \t,\cos (\tfrac{p \pi}{k})\big)$, which 
is the stated result once we rewrite the final formula in $t = \cos \t$. 
\end{proof}

Using the fact that $p_{n}\big(w_{\l-\f12}^{(k),1}; - t \big)$ is an orthogonal polynomial with respect
to $w_{\l-\f12}^{(k),-1}$, we could derive a similar formula for $ \phi_r 
\big(w_{\l-\f12}^{(k),-1}; t, \cos \left(\pi - \tfrac{2 p \pi}{k}\right)\big)$. We shall do so only for the 
case when $p=0$, which is stated in the corollary below. 

\begin{cor} \label{cor:poisson3}
For $k=2,3,4\ldots$ and $0 \le r <1$, 
\begin{align} \label{eq:poisson3B}
 \sum_{n=0}^\infty 
 \frac{p_{n}\big(w_{\l-\f12}^{(k),1}; t\big)p_{n}\big(w_{\l-\f12}^{(k),1}; 1\big)} {h_n \big(w_{\l-\f12}^{(k),1}\big)} r^n 
      =   \frac{1-r}{(1- 2 r t + r^2) (1-2 r^k T_k (t) + r^{2k})^{\l} },  
\end{align}
and, furthermore, 
\begin{align} \label{eq:poisson3C}
  \sum_{n=0}^\infty 
   \frac{p_{n}\big(w_{\l-\f12}^{(k),-1}; t\big)p_{n}\big(w_{\l-\f12}^{(k),-1}; -1\big)} {h_n \big(w_{\l-\f12}^{(k),-1}\big)} r^n
      =           \frac{1-r}{(1+ 2 r t + r^2) (1-2 r^k T_k (-t) + r^{2k})^{\l} }. 
\end{align}
\end{cor}

\begin{proof}
The identity \eqref{eq:poisson3B} is the case $p= 0$ of the identity in Theorem \ref{thm:poisson3}. The
identity  \eqref{eq:poisson3C} follows from the fact that $p_{n}\big(w_{\l-\f12}^{(k),1}; - t \big)$ is an 
orthogonal polynomial with respect to $w_{\l-\f12}^{(k),-1}$ and $h_n(w_{\l-\f12}^{(k),1}) = h_n(w_{\l-\f12}^{(k),-1})$. 
\end{proof}

In the case $k=1$, these formulas give simple generating functions for the Jacobi polynomials 
$P_n^{(\l+\f12,\l-\f12)}(t)$ and $P_n^{(\l-\f12,\l+\f12)}(t)$, respectively, which we state below:
\begin{align}\label{eq:JacobiGenerat1}
 \sum_{n=0}^\infty  \frac{(2n+2\l+1) (2\l+1)_n}{(2\l+1)(\l+\f32)_n} P_n^{(\l-\f12,\l+\f12)}(t)  r^n 
   &\,  =   \frac{1-r}{(1- 2 r t+ r^2)^{\l+1} }, \\
 \sum_{n=0}^\infty  \frac{(2n+2\l+1) (2\l+1)_n}{(2\l+1)(\l+\f32)_n} P_n^{(\l+\f12,\l-\f12)}(t)  r^n 
     &\, =   \frac{1+r}{(1-2 r t+ r^2)^{\l+1} },   \label{eq:JacobiGenerat2}
\end{align}
In the case of the second identity, we have used $P_n^{(\a,\b)}(-1) = (-1)^n P_n^{(\b,\a)}(1)$
and replaced $r$ by $-r$. As far as we are aware, these identities are new.

\section{Product formula and intertwining operator}
\setcounter{equation}{0}

The generating functions that we derived in the previous sections lead to several integral representations
of the corresponding orthogonal polynomials. We state one such result as an example. 

\begin{thm}
Let $\l \ge 0$. For $k= 1, 2, 3, \ldots,$ and $n = 0,1,2,\ldots$, 
\begin{align} \label{eq:k-Gegen}
 & \frac{p_n \big(w_{\l-\f12}^{(k)}; \cos \t \big)p_n \big(w_{\l-\f12}^{(k)}; 1 \big)}{h_n\big(w_{\l-\f12}^{(k)}\big)} 
    = \frac{n + k \l}{k \l} a_\l^{(k)} \\
  & \qquad\qquad\qquad
     \times  \int_{T^{k-1}} C_n^{k\l} \Bigg( \sum_{j=0}^{k-1} \cos \left(\t - \frac{2 j\pi}{k}\right) u_j \Bigg) 
       u_0^\l \prod_{i=1}^{k-1} u_i^{\l-1}\dd u. \notag
\end{align}
In particular, for $m = 0,1,2,\ldots$, 
\begin{align} \label{eq:k-Gegen2}
 C_m^{\l}(\cos (k\t) ) = a_\l^{(k)} \int_{T^{k-1}}  & 
   C_{km}^{k\l} \Bigg( \sum_{j=0}^{k-1} \cos \left(\t - \frac{2 j\pi}{k}\right) u_j \Bigg) u_0^\l \prod_{i=1}^{k-1} u_i^{\l-1}\dd u. 
\end{align}
\end{thm} 

The integral representation \eqref{eq:k-Gegen} follows from \eqref{eq:generating1}, 
Lemmas \ref{lem:integral} and \ref{lem:trig-identity}, and \eqref{eq:Gegen2}. Setting $n = k m$ in
\eqref{eq:k-Gegen} gives \eqref{eq:k-Gegen} by Proposition \ref{prop:GGk} and \eqref{eq:Gegen-recur}. 

When $k =2$, the identity \eqref{eq:k-Gegen2} is a special case of the identity  \cite[Theorem 1.5.6]{DX}
\begin{align} \label{eq:GGproduct0}
  C_m^{(\l,\mu)}(t) = c_\mu \int_{-1}^1 C_m^{\l+\mu} (t u) (1+u)(1-u^2)^{\mu-1} \dd u
\end{align}
for the generalized Gegenbauer polynomials defined in \eqref{eq:GGpoly}, since it can be easily verified that
$C_{2m}^{(\l,\l)}(\cos\t) = C_m^\l (\cos 2 \t)$ and  $\cos \t u_0 + \cos (\t - \pi) u_1 = \cos \t (1- 2 u_1)$ with 
$u_0 = 1-u_1$. Furthermore, the generalized Gegenbauer polynomials satisfy a product formula \cite[(2.10)]{X97} 
\begin{align} \label{eq:GGproduct}
  & \frac{C_m^{(\l,\mu)}(\cos \t) C_m^{(\l,\mu)}(\cos \phi)} {C_m^{(\l,\mu)}(1)} = \frac{m+\l+\mu}{\l+\mu}   c_\l c_\mu \\
   & \times 
    \int_{-1}^1\int_{-1}^1 C_m^{\l+\mu} (t \cos \t \cos \phi + s \sin \t \sin \phi) (1+t) (1-t^2)^{\mu-1} (1-s^2)^{\l-1} \dd t \dd s, \notag
\end{align}
which reduces to, when $m =2n$, the product formula for the Jacobi polynomials proved in \cite[p.192, (2.5)]{DK}, 
see also \cite[p.133]{K}. Evidently \eqref{eq:GGproduct} becomes \eqref{eq:GGproduct0} when $\phi =0$. The 
identity \eqref{eq:GGproduct} can be deduced from the identity \eqref{eq:h-zonal} for the $h$-harmonics and the 
closed formula of the intertwining operator for the group $I_2 = \ZZ_2 \times \ZZ_2$ as shown in \cite{X97}. 

For the dihedral group, the identity \eqref{eq:h-zonal} gives a product formula for the $h$-harmonics, from which 
we can deduce a product formulas for orthogonal polynomials $p_n(w_{\l-\f12}; \cdot)$. The identities \eqref{eq:k-Gegen}
and \eqref{eq:k-Gegen2} are consequences of our partial closed form of the intertwining operator in 
Theorem \ref{thm:main}. A full closed form of the intertwining operator would lead to a product formula for
$p_n(w_{\l-\f12}; \cdot)$ that will be a generalization of \eqref{eq:GGproduct}.  In view of \eqref{eq:GGproduct0}
and \eqref{eq:GGproduct}, however, a full closed form could be much more involved and is still not known at
this point.


\begin{thebibliography}{99}
 
\bibitem{AAA}
        W. Al-Salam, W. R. Allaway and R. Askey, 
        Sieved ultraspherical polynomials.
        \textit{Trans. Amer. Math. Soc.} \textbf{284} (1984), 39--55.
         
\bibitem{Bie}
        D. Constales, H. De Bie, P. Lian, Explicit formulas for the Dunkl dihedral kernel and the $(\kappa,a)$-generalized 
        Fourier kernel. 
        \textit{J. Math. Anal. Appl.} \textbf{460} (2018), 900--926. 
         
\bibitem{DaX} 
        F. Dai and Y. Xu, 
        \textit{Approximation theory and harmonic analysis on spheres and balls}. 
        Springer Monographs in Mathematics, Springer, 2013. 

\bibitem{DDY}
        L. Deleaval, N. Demni, H.Youssfi,
        Dunkl kernel associated with dihedral groups.
        \textit{J. Math. Anal. Appl.} \textbf{432} (2015), 928--944. 

\bibitem{DK}
        A. Dijksma and T. Koornwinder, 
        Spherical harmonics and the product of two Jacobi polynomials. 
        \textit{Indag. Math.} \textbf{33} (1971), 191--196.

\bibitem{D89}
        C. F. Dunkl, 
        Differential-difference operators associated to reflection groups. 
        \textit{Trans. Amer. Math. Soc.}, \textbf{311} (1989), 167--183.

\bibitem{D92}
        C. F. Dunkl,  
        Poisson and Cauchy kernels for orthogonal polynomials with dihedral symmetry.
        \textit{J. Math. Anal. Appl.} \textbf{143} (1989), 459--470. 
   
\bibitem{D95}
        C. F. Dunkl, 
        Intertwining operators associated to the group $S^3$.
        \textit{Trans. Amer. Math. Soc.} \textbf{347} (1995), 3347--3374.

\bibitem{D-B2}
        C. F. Dunkl, 
        An intertwining operator for the group $B_2$. 
        \textit{Glasgow Math. J.} \textbf{49}, 2007, 291--319.

\bibitem{DX}  
       C. F. Dunkl and Y. Xu, 
       \textit{Orthogonal polynomials of several variables, 2nd ed.}
       Encyclopedia of Mathematics and its Applications \textbf{155}, Cambridge
       University Press, 2014.
        
\bibitem{K}
        T. Koornwinder, 
        Jacobi polynomials, II. An analytic proof of the product formula. 
        \textit{SIAM J. Math. Anal.} \textbf{5} (1974), 125--137.


\bibitem{Sz}
       G. Szeg\H{o},
       \textit{Orthogonal polynomials}, 4th edition.
       Amer. Math. Soc., Providence, RI. 1975. 

\bibitem{X97}
       Y. Xu,
       Orthogonal polynomials for a family of product weight functions on the spheres.
       \textit{Canad. J. Math.} \textbf{49} (1997), 175--192.

\bibitem{X00}
        Y. Xu,
        A product formula for Jacobi polynomials. 
        \textit{Special functions (Hong Kong, 1999)}, 423--430, World Sci. Publ., River Edge, NJ, 2000. 

\bibitem{X15} 
        Y. Xu, 
        An integral identity with applications in orthogonal polynomials. 
        \textit{Proc. Amer. Math. Soc.} \textbf{143} (2015), 5253--5263.

\end{thebibliography}
\end{document}